\documentclass[11pt, reqno]{amsart}
\usepackage[utf8]{inputenc}
\usepackage{amsmath}
\usepackage{amsthm}
\usepackage{amsfonts}
\usepackage{amssymb}
\usepackage{mathrsfs, mathtools}
\usepackage[colorlinks=true,linkcolor=blue,citecolor=blue,urlcolor=blue]{hyperref}

\usepackage{enumerate, xcolor, graphicx}
\definecolor{Magenta}{cmyk}{0,1,0,0}
\definecolor{dgreen}{rgb}{0.,0.6,0.}

\usepackage[top=25mm,bottom=25mm,left=27mm,right=27mm]{geometry}

\newtheorem{theorem}{Theorem}[section]
\newtheorem{proposition}{Proposition}[section]
\newtheorem{conjecture}{Conjecture}[section]

\newtheorem*{definition*}{Definition}
\newtheorem*{theorem*}{Theorem}
\newtheorem*{remark*}{Remark}
\newtheorem*{problem*}{Problem}
\newtheorem*{conjecture*}{Conjecture}
\newtheorem{lemma}{Lemma}[section]

\newcommand{\Sone}{\mathbb S^1}
\newcommand{\e}{\varepsilon}

\begin{document}
	
	\title[On conjectures of B. Berndt and A. Straub]{On Berndt-Straub conjectures about \\ zeros of generalized Ramanujan polynomials}
	

	\author[S. Pathak]{Siddhi Pathak}
	
	

	\address{(S. Pathak) Chennai Mathematical Institute, H-1 SIPCOT IT Park, Siruseri, Kelambakkam, Tamil Nadu, India 603103.}
	
	\date{\today}
	
	\email{siddhi@cmi.ac.in}
	
	\subjclass[2020]{11M06, 11B68, 30C15}
\keywords{Generalized Ramanujan polynomial, Dirichlet characters, unimodularity of zeros}
	
	\thanks{Research of the author was partially supported by an INSPIRE Faculty fellowship.}

        \begin{abstract}
We settle two conjectures of Berndt and Straub asserting that all nonzero
zeros of the generalized Ramanujan polynomials $R_k(x;\chi,\psi)$ and
$S_{2n}(x;\chi,\chi)$ lie on the unit circle for real nonprincipal
Dirichlet characters.
We prove their conjecture for $R_k$.
For $S_{2n}$, we prove that at most four nonzero zeros lie off the unit
circle, all purely imaginary. If
$\chi_{163}=\left(\frac{\cdot}{163}\right)$ is the primitive quadratic
character modulo $163$, then $S_6(x;\chi_{163},\chi_{163})$ has four
exceptional zeros, showing that this bound is sharp.
This disproves the second conjecture and establishes that the \textit{`Riemann
hypothesis for period polynomials'} can fail for Eisenstein newforms.
Further, we introduce a transcendental invariant whose sign
determines whether exceptional zeros of $S_{2n}(x;\chi,\chi)$ occur in all sufficiently large
degrees, and show that these zeros converge to transcendental limits
off the unit circle.
For primitive odd characters, we relate this invariant and these
limits to Fekete polynomials.
\end{abstract}

	\maketitle
	
	\section{\bf Introduction}
	\bigskip

    Let $f\in S_k(\Gamma_0(N))$ be a newform of even weight $k\geq4$ and trivial Nebentypus. The `period polynomial' associated with $f$ is defined as the period integral
\[
 r_f(X)=\int_0^{i\infty}f(\tau)(\tau-X)^{k-2}\,d\tau,
\]
 and has the expansion
\[
 r_f(X)=
 \sum_{j=0}^{k-2}\binom{k-2}{j}
 i^{j+1}\frac{\Gamma(j+1)}{(2\pi)^{j+1}}
 L(f,j+1)(-X)^{k-2-j};
\]
see \cite[Section 1, equations (1.1)--(1.3)]{JMOS}. Here $L(f,s)$ denotes the $L$-function associated with the cuspform $f$. \\

The functional equation of $L(f,s)$, namely,
$\Lambda(f,s)=\varepsilon(f)\Lambda(f,k-s)$, implies that $r_f$ satisfies
\[
r_f(z)
=-i^k\, \varepsilon(f)\, (\sqrt{N}\,z)^{k-2} \, \, 
r_f\!\left(-\frac{1}{Nz}\right).
\]
Therefore, nonzero
zeros of $r_f$ are invariant under $\alpha\mapsto-1/(N\alpha)$, suggesting that all its zeros lie
on the circle $|X|=N^{-1/2}$. For level one, Conrey,
Farmer and Imamo\u glu \cite{CFI} established unimodularity of nontrivial zeros for the odd part of $r_f$, while El-Guindy and Raji \cite{EGR}
proved it for the full period polynomials. Jin, Ma,
Ono and Soundararajan \cite[Theorem 1.1]{JMOS}
subsequently established the assertion for
newforms described above with arbitrary level and termed it the `Riemann hypothesis for period polynomials'. \\

For Eisenstein series, the period integral above requires
\textit{regularization} in order to converge because of the nonzero constant terms at
the cusps. The resulting period function is a Laurent
polynomial with coefficients in terms of
special values of the associated $L$-function (see \cite[Section 5]{BS}). For the level-one Eisenstein series of even weight $k\geq4$, the odd part of the corresponding period function is, up to
a scalar factor, 
\[
 \sum_{j=0}^{k/2}
 \frac{B_{2j}}{(2j)!}
 \frac{B_{k-2j}}{(k-2j)!}\,X^{2j-1},
\]
where $B_n$ denotes the $n$-th Bernoulli number \cite[equation (35)]{BS}. Multiplying throughout by $X$, we obtain a polynomial, which is more commonly known as a \textit{Ramanujan polynomial}, introduced in \cite{MSW}, as it connects the period relation with Ramanujan's formula for odd zeta values. \\

In the same paper, Murty, Smyth and Wang \cite{MSW} proved that all except four zeros of these Ramanujan polynomials lie on the unit circle; the four exceptions are real and form two reciprocal pairs. Lal\'in and Smyth \cite{LS} established unimodularity for zeros of the polynomial obtained from the full level-one Eisenstein period function by multiplication by $X$. For certain Eisenstein series on the Fricke groups $\Gamma_0^+(N)$, Choi and Im \cite{ChoiIm} proved that all zeros of the associated period functions lie on the circle $|X|=N^{-1/2}$. More recently, Fichter, Shah and Wong \cite{FSW} proved that the even-indexed members of a family of Ramanujan-like polynomials related to double Mordell–Tornheim series have all their zeros on the unit circle. In \cite{BS}, Berndt and Straub discussed this framework in the context of Eisenstein series attached to Dirichlet characters through their study of Eichler integrals. \\


Let $\chi$ and $\psi$ be \textit{primitive} nonprincipal characters
of conductors $L$ and $M$, respectively, and suppose that
$k\geq4$ and $\chi(-1)\psi(-1)=(-1)^k$. Following \cite[Section 4, equation (15)]{BS}, set
\[
 E_k(\tau;\chi,\psi)
 =\sum_{\substack{(m,n)\in\mathbb Z^2\\(m,n)\ne(0,0)}}
   \frac{\chi(m)\psi(n)}{(m\tau+n)^k},
\]
so that $E_k(M\tau;\chi,\psi) \in M_k(\Gamma_0(LM),\chi\overline{\psi})$. Define the period polynomial of $E_k(\tau;\chi,\psi)$ by
\[
 r_{k;\chi,\psi}(X)
 :=\int_0^{i\infty}
 E_k(\tau;\chi,\psi)(\tau-X)^{k-2}\,d\tau.
\]
The nonprincipality of $\chi$ makes the constant term
at $\infty$ vanish. The transformation identity
\[
 \tau^{-k}E_k(-1/\tau;\chi,\psi)
 =\chi(-1)E_k(\tau;\psi,\chi)
\]
and the nonprincipality of $\psi$ give the corresponding
vanishing at $0$. Hence the integral converges absolutely
and defines a polynomial of degree at most $k-2$.
Its coefficients admit the explicit evaluation
\cite[Theorem 6.1]{BS}
\[
 r_{k;\chi,\psi}(X)
 =-\frac{4\psi(-1)}{k-1}
   \sum_{\substack{1\leq j\leq k-1\\
                   (-1)^j=\chi(-1)}}
   L(j,\chi)L(k-j,\psi)X^{k-j-1}.
\]

For a nonprincipal character $\eta$ modulo $q$, the generalized Bernoulli numbers $B_{j,\eta}$ are defined via the series
\[
 \sum_{a=1}^{q}\eta(a)\frac{te^{at}}{e^{qt}-1}
 =\sum_{j\geq0}B_{j,\eta}\frac{t^j}{j!}.
\]
In their work on Eichler integrals of Eisenstein series, Berndt and Straub \cite{BS} considered the two families of polynomials for nonprincipal characters $\chi$ mod $L$ and $\psi$ mod $M$:
\begin{align}
 S_k(X;\chi,\psi)
 &=\sum_{s=0}^{k}
 \frac{B_{s,\chi}}{s!}\frac{B_{k-s,\psi}}{(k-s)!}
 \left(\frac{LX}{M}\right)^{k-s-1},
   \label{Sdef}\\
 R_k(X;\chi,\psi)
 &=\sum_{s=0}^{k}
 \frac{B_{s,\chi}}{s!}\frac{B_{k-s,\psi}}{(k-s)!}
 \left(\frac{X-1}{M}\right)^{k-s-1}(1-X^{s-1}).
   \label{Rdef}
\end{align}
The extremal terms in the original definitions vanish
because $B_{0,\chi}=B_{0,\psi}=0$. Moreover,
$B_{j,\chi}=0$ unless $\chi(-1)=(-1)^j$, so a nonzero
polynomial requires
\[
 \chi(-1)\psi(-1)=(-1)^k.
\]
The two families satisfy
\[
 R_k(X;\chi,\psi)
 =
 S_k\!\left(\frac{X-1}{L};\chi,\psi\right)
 -
 X^{k-2}
 S_k\!\left(\frac{X-1}{LX};\chi,\psi\right).
\]\\

Recall that if $\eta$ is primitive of conductor
$q$ and $j\geq1$ is an integer satisfying
$\eta(-1)=(-1)^j$, one has
\[
 L(j,\eta)
 =(-1)^{j-1}\frac{G(\eta)}{2}
   \left(\frac{2\pi i}{q}\right)^j
   \frac{B_{j,\overline{\eta}}}{j!},
\]
as in \cite[equation (23)]{BS}. Here $G(\eta)$ denotes the associated Gauss sum. Thus, one obtains the relation that 
\[
 r_{k;\chi,\psi}(X)
 =- \left( \frac{\chi(-1)G(\chi)G(\psi)(2\pi i)^k}
         {(k-1)L^{k-1}M} \right) \,
   S_k(X;\overline{\chi},\overline{\psi}).
\]
Thus the zeros of $S_k$ are precisely those of the
corresponding Eisenstein period polynomial. While these interpretations assume primitivity of
characters, the Bernoulli-numbers definitions of both families extend to arbitrary Dirichlet characters. Motivated by these relations and numerical evidence, Berndt and Straub formulated the following two conjectures on the location of their zeros.

\begin{conjecture}[Berndt--Straub, {\cite[Conjecture 7.3]{BS}}]\label{intro:conjR}
Let $\chi$ and $\psi$ be nonprincipal real Dirichlet characters. If $R_k(x;\chi,\psi)$ is nonzero, then every zero of $R_k(x;\chi,\psi)$ lies on $\Sone$.
\end{conjecture}

\begin{conjecture}[Berndt--Straub, {\cite[Conjecture 7.9]{BS}}]\label{intro:conjS}
Let $\chi$ be a nonprincipal real Dirichlet character. If $S_{2n}(x;\chi,\chi)$ is nonzero, then every nonzero zero of $S_{2n}(x;\chi,\chi)$ lies on $\Sone$.
\end{conjecture}

Our main theorems resolve the above two conjectures as follows. 

\begin{theorem}\label{intro:main-R-k}
Conjecture~\ref{intro:conjR} holds. 
\end{theorem} 

\begin{theorem}\label{intro:main-S-2n}
 For every nonprincipal real character $\chi$ and every nonzero polynomial $S_{2n}(x;\chi,\chi)$, at most four nonzero zeros lie off $\Sone$. If such zeros occur, they form a purely imaginary reciprocal quartet
\[
 \pm i\sqrt\rho,\qquad \pm i/\sqrt\rho,\qquad \rho>1.
\]
If $\chi$ is a primitive even real character, then $S_{2n}(x;\chi,\chi)$ has all its nonzero zeros on the unit circle.
\end{theorem}

Furthermore, we establish that Conjecture~\ref{intro:conjS} is false in general with the following counterexample.\\
\textbf{Example.} Let  $\chi_{163}=(\frac{\cdot}{163})$ be the primitive quadratic character mod $163$. The finite formula
\[
B_{j,\chi_{163}}=163^{j-1}\sum_{a=1}^{162}\chi_{163}(a)B_j(a/163)
\]
gives
\[
B_{1,\chi_{163}}=-1,\qquad B_{3,\chi_{163}}=13890,\qquad
B_{5,\chi_{163}}=-213225610.
\]
Substitution in \eqref{Sdef} yields
\begin{equation}\label{intro:example}
 S_6(x;\chi_{163},\chi_{163})
 =\frac1{12}\bigl(21322561x^4+64310700x^2+21322561\bigr).
\end{equation}
The middle coefficient above exceeds twice the outer coefficient. Consequently the reciprocal quadratic in $y=x^2$ has two negative reciprocal roots distinct from $-1$, and all four zeros in $x$ are off the unit circle. Thus, Conjecture~\ref{intro:conjS} fails to hold in general, even if $\chi$ is a \textit{primitive} odd character. \\

This example also shows that the Riemann hypothesis for period
polynomials does not extend to all Eisenstein newforms.
Indeed, $E_6(163\tau;\chi_{163},\chi_{163})$ is, up to
normalization, an Eisenstein newform of level $163^2$ and
trivial Nebentypus, and an eigenform for the full Hecke algebra
\cite[Section~2]{LinowitzThompson}.
Its full period polynomial is a nonzero scalar multiple of
$S_6(163X;\chi_{163},\chi_{163})$ and therefore has four zeros
off the expected circle $|X|=1/163$.
Thus the failure occurs even for primitive characters
and a Hecke eigenform that is `new' at its level.
This contrasts with the theorem for cuspidal newforms
\cite{JMOS} and the unimodularity result for full level-one
Eisenstein period functions \cite{LS}.\\


Our proofs combine a Chebyshev sign argument with explicit bounds for Dirichlet $L$-values. For $R_k$, the sign argument accounts for all but at most two zeros, while a decomposition into geometric sums rules out real zeros off the unit circle. For $S_{2n}$, we apply the same argument after the substitution $y=x^2$. In this case, a reciprocal pair of negative zeros may occur, with its existence determined by the sign of an appropriate endpoint value or derivative. This distinction explains the different conclusions for the two families. We give direct proofs of all coefficient estimates.\\

The counterexample above also raises the question of whether, for a fixed
character $\chi$, zeros off the unit circle occur in infinitely many
degrees. For primitive even characters, and for odd characters with
$B_{1,\chi}=0$, all nonzero zeros of every nonzero polynomial
$S_{2n}(x;\chi,\chi)$ lie on $\Sone$. For the remaining characters, to determine whether exceptional zeros persist as the degree
increases, we introduce a function whose value at $1$ governs
their eventual occurrence. This function also has arithmetic
properties of independent interest.\\

Let $\chi$ be a real nonprincipal character modulo $N$,
induced by the primitive character $\chi^*$ of conductor $f$.
Assume that $B_{1,\chi}\neq0$ if $\chi$ is odd, and that
$\chi$ is imprimitive if $\chi$ is even. Put $\e=1$ or $2$,
respectively, and define
\[
\alpha_r=
L(2r+\e,\chi^*)
\prod_{\substack{p\mid N\\p\nmid f}}
\bigl(1-\chi^*(p)p^{1-2r-\e}\bigr),
\qquad
Q=f\prod_{\substack{p\mid N\\p\nmid f}}p.
\]
Consider
\[
\Phi_\chi(t)=\frac1{1+t}
+\sum_{r=0}^{\infty}(\alpha_r-1)(-t)^r
\qquad \text{ and set } \qquad 
I_\chi:=\Phi_\chi(1)
=\frac12+\sum_{r=0}^{\infty}(-1)^r(\alpha_r-1).
\]
The following theorem establishes convergence and describes
the arithmetic properties of this function.
\begin{theorem}\label{arithmeticlimit}
Assume the notation and hypotheses above. Then the power series
\[
\sum_{r=0}^{\infty}(\alpha_r-1)(-t)^r
\]
converges absolutely for $|t|<4$. Consequently,
$\Phi_\chi$ is meromorphic in the disk $|t|<4$, with a unique simple pole at $t=-1$ with residue $1$.
Moreover:
\begin{enumerate}[(i)]
\item The number $I_{\chi} = \Phi_{\chi}(1)$ and $\pi$ are algebraically independent, and hence, $I_{\chi}$ is transcendental. In particular, $I_{\chi} \neq 0$.

\item If $I_\chi>0$, then $\Phi_\chi$ has no zeros in
$|t|<1$. 

\item If $I_\chi<0$, then $\Phi_\chi$ has exactly one zero
in $|t|<1$, denoted by $t_\chi$. This zero is simple
and lies in $(0,1)$. Moreover,
$\exp(2\pi\sqrt{t_\chi}/Q)$ is algebraic, so
$t_\chi$ is transcendental by the Gelfond--Schneider
theorem.
\end{enumerate}
\end{theorem}

The sign of $I_\chi$ gives the following complete
description of the eventual occurrence of exceptional
zeros. When they occur, their limits are determined by
the transcendental number $\rho_\chi:= t_{\chi}^{-1}$.
\begin{theorem}\label{eventualzeros}
Let $\chi$ be a real nonprincipal character modulo $N$.
Assume that $B_{1,\chi}\neq0$ if $\chi$ is odd, and that
$\chi$ is imprimitive if $\chi$ is even.
\begin{enumerate}[(i)]
\item If $I_\chi>0$, then for every sufficiently large $n$,
all nonzero zeros of $S_{2n}(x;\chi,\chi)$ lie on $\Sone$.

\item If $I_\chi<0$, then for every sufficiently large $n$,
the nonzero zeros of $S_{2n}(x;\chi,\chi)$ off $\Sone$ are precisely
the four simple zeros
\[
\pm i\sqrt{\rho_n},\qquad
\pm i/\sqrt{\rho_n},\qquad \rho_n>1.
\]
Moreover, $\rho_n\to\rho_\chi$, and for every fixed $R$ with 
$\max\{1/4,\rho_\chi^{-1}\}<R<1$,
\[
\rho_n-\rho_\chi=O_{\chi,R}(R^n).
\]
\end{enumerate}
\end{theorem}

When $\chi$ is a primitive odd character, $\Phi_\chi$ can be expressed directly through a Fekete polynomial. In particular, the sign of $I_\chi$ is determined by a value of a finite character sum. We establish this in Proposition \ref{Feketerelation}.

\section{\bf Preliminaries}
\medskip

\subsection{Values of \texorpdfstring{$L$}{L}-functions of imprimitive characters}
Let $\xi$ be a nonprincipal real Dirichlet character mod $N$, induced from a primitive character $\xi^*$ of conductor $f$. Set
\[
 T_\xi=\{p:p\mid N,\ p\nmid f\},\qquad Q_\xi=f\prod_{p\in T_\xi}p,
\]
and, for positive integers $m$, define
\begin{equation}\label{Vdef}
 V_m(\xi)=L(m,\xi^*)\prod_{p\in T_\xi}\left( 1- \frac{\xi^*(p)}{p^{m-1}} \right).
\end{equation}
For $m>1$ this is positive. For odd $\xi$, $V_1\geq0$ and it is zero exactly when some $p\in T_\xi$ satisfies $\xi^*(p)=1$. We use the standard positivity $L(1,\xi^*)>0$ for real nonprincipal primitive characters.

\begin{lemma}\label{berL}
For positive integers $m$ with $\xi(-1) = (-1)^m$,
\begin{equation}\label{Bernabs}
 \frac{|B_{m,\xi}|}{m!}
 =\frac{2}{\sqrt f\prod_{p\in T_\xi}p}
 \left(\frac{Q_\xi}{2\pi}\right)^m V_m(\xi).
\end{equation}
For $m$ with $\xi(-1) = (-1)^m$, the numbers $B_{m,\xi}$ and $B_{m+2,\xi}$ have opposite signs whenever both are nonzero. In particular, for fixed $m$, all nonzero products $B_{j,\xi} \,B_{m-j,\xi}$, with $1\leq j<m$, have the same sign. More generally, for fixed real nonprincipal characters $\xi,\eta$ and fixed $k$, all nonzero products $B_{s,\xi}B_{k-s,\eta}$, with $1\leq s<k$, have the same sign.
\end{lemma}
\begin{proof}
By the special-value formula \cite[equation (23)]{BS}, we have
\[
 B_{m,\xi^*}
 =\frac{2(-1)^{m-1}m!f^m}
        {(2\pi i)^m \, G(\xi^*)}L(m,\xi^*).
\]
Here $G(\xi^*)$ is the associated Gauss sum. Since $\xi^*$ is real and $\xi^*(-1)=(-1)^m$,
the identity $G(\xi^*)^2=(-1)^m f$ gives
\[
\frac{B_{m,\xi^*}}{m!}
 =-\frac{2 \, f^{m-1}G(\xi^*)}
         {(2\pi i)^m}  \,L(m,\xi^*).
\]
When $\xi$ is induced from $\xi^*$, we have
\[
 B_{m,\xi}=B_{m,\xi^*}\prod_{p\in T_\xi}(1-\xi^*(p)p^{m-1}).
\]
The claims follow by putting together the above two expressions.
\end{proof}

\begin{lemma}\label{ratios}
Let $\xi$ be as before and $v>u\geq3$. Then
\begin{equation}\label{Vratio}
 \frac{V_v(\xi)}{V_u(\xi)}\leq\frac{\zeta(u-1)}{\zeta(u)}
 \leq A:=\frac{\zeta(2)}{\zeta(3)}.
\end{equation}
If $r\geq1$, $D\geq2r+2$, and $\xi(-1) = (-1)^{D-2r}$, then, except when $f=3$ and $T_\xi=\varnothing$,
\begin{equation}\label{scaledratio}
 Q_\xi^{-2r}\frac{V_{D-2r}(\xi)}{V_D(\xi)}
 \leq \frac{\zeta(2)\zeta(4)}{16^r}.
\end{equation}
If $f=3$ and $T_{\xi}={\varnothing}$, then the left side is at most $9^{-r}$.
\end{lemma}
\begin{proof}
For \(v\ge u>1\), we compare each Euler factor at \(v\) with its value at \(u\) considering the possibilities $\xi^*(p) = \pm 1$. The factors
\[
(1-p^{-s})^{-1},
\qquad
\frac{1+p^{1-s}}{1+p^{-s}}
\]
decrease with \(s\). Thus, their ratios are at most \(1\). For the $L$-value factor with $\xi^*(p) = -1$,
\[
\frac{1+p^{-u}}{1+p^{-v}}
\le 1+p^{-u}
\le \frac{1-p^{-u}}{1-p^{1-u}},
\]
where the last inequality follows upon clearing denominators from
\((p-2)p^{-u}+p^{1-2u}\ge0\). For the correction factor with $\xi^*(p) = 1$,
\[
\frac{(1-p^{1-v})/(1-p^{-v})}
     {(1-p^{1-u})/(1-p^{-u})}
\le \frac{1-p^{-u}}{1-p^{1-u}},
\]
because its numerator is at most \(1\). Thus every ratio of Euler factors satisfies the same bound. Multiplying these bounds over all primes gives, for \(u>2\),
\[
\frac{V_v(\xi)}{V_u(\xi)}
\le\prod_p\frac{1-p^{-u}}{1-p^{1-u}}
=\frac{\zeta(u-1)}{\zeta(u)}.
\]
Since this bound is a decreasing function of $u$, we derive the second inequality.\\

\noindent For \eqref{scaledratio}, set $s = D-2r$ and write
\begin{equation}\label{scaledlocal}
 Q_\xi^{-2r}\frac{V_s(\xi)}{V_D(\xi)}
 =f^{-2r}\frac{L(s,\xi^*)}{L(D,\xi^*)}
 \prod_{p\in T_\xi}p^{-2r}
 \left( \frac{1-\xi^*(p) \, p^{1-s}}{1-\xi^*(p) \, p^{1-D}} \right).
\end{equation}
Each {scaled} Euler factor in this product is at most $1$. Also
\[
 L(s,\xi^*)\leq\zeta(s)\leq\zeta(2),\qquad
 \frac1{L(D,\xi^*)}\leq\frac{\zeta(D)}{\zeta(2D)}\leq\zeta(4),
\]
since $D\geq4$. Thus we obtain the bound for $f\geq4$. If $f=3$, then $\xi^*$ is odd and $s\geq3$. If $T_\xi\ne\varnothing$, choose one omitted prime $p$. Its scaled factor is at most
$(1+p^{1-s})p^{-2r}\leq(5/4)4^{-r}$. Hence the conductor and this factor together give
$(5/4)36^{-r}\leq16^{-r}$; all remaining scaled factors are at most $1$. If $f=3$ and $T_\xi=\varnothing$, the relevant primitive $L$-values increase along odd integers $s\geq3$. Indeed,
\[
 L(s+2,\xi^*)-L(s,\xi^*)
 \geq \frac{3}{2^{s+2}}-\sum_{n\geq4}\frac{1}{n^{s}}>0,
\]
because $\sum_{n\geq4} \frac{1}{n^{s}}<\frac{1}{4^{s}}+\frac{4^{1-s}}{(s-1)}\leq \frac{3}{4^{s}}\leq \frac{3}{2^{s+2}}$. This yields the bound $9^{-r}$.
\end{proof}

\subsection{\bf A unimodularity criterion}
Let $A$, $a_r \in \mathbb{R}$ for $1 \leq r \leq d-1$, for an integer $d \geq 2$. Suppose that $a_r = a_{d-r}$ for all $1 \leq r \leq d-1$.  Consider 
\begin{equation}\label{Hform}
 H(y)=A(1+y^d)+\sum_{r=1}^{d-1}a_ry^r.
\end{equation}
Note that we allow $A=0$. Set $F(\theta)=e^{-id\theta/2} \, H(e^{i\theta})$ and define
\[
 \theta_j=\frac{(2j-1)\pi}{d} \qquad \text{and} \qquad \sigma_j=(-1)^{j-1},\qquad 1\leq j\leq m=\lfloor d/2\rfloor.
\]
Thus, we have
\[
\cos(d\theta_j/2)=0,
\qquad
\sin(d\theta_j/2)=\sigma_j.
\]
Pairing symmetric terms and evaluating at $\theta = \theta_j$ gives
\begin{equation}\label{nodeS}
 \sigma_j \,F(\theta_j)=
 \begin{cases}
 \sigma_j \, a_m+2\sum_{r=1}^{m-1}a_r \, \sin(r\theta_j),&d=2m,\\
 2\sum_{r=1}^{m}a_r \, \sin(r\theta_j),&d=2m+1.
 \end{cases}
\end{equation}
In particular, the outer coefficient $A$ disappears in the above specialization. 

\begin{lemma}\label{angular}
For a given $H$ as in \eqref{Hform}, if the expression on the right of \eqref{nodeS} is positive for every $1\leq j\leq m$, then $H$ has at least $d-2$ distinct zeros on $\Sone$. More specifically, if
\begin{align}
 4\sum_{r=1}^{m-1}r|a_r-1|+\cot\left(\frac\pi{4m}\right)|a_m-1|&<1 &&(d=2m),\label{eventest}\\
 4\sum_{r=1}^{m}r|a_r-1|&<1 &&(d=2m+1),\label{oddtest}
\end{align}
then $H$ has at least $d-2$ distinct zeros on $\Sone$.
\end{lemma}
\begin{proof}
The hypothesis implies that $F(\theta_j)$ alternate in sign on the upper semicircle of the unit circle. Thus, by the intermediate value theorem, we obtain $m-1$ zeros of $H$ there, and their $m-1$ conjugates on the lower semicircle. Further if $d$ is odd, then $H(-1)=0$, giving a total of $d-2$ zeros in either case. \\

The identity
\[
2\sum_{r=1}^{q}\sin(r\theta)
=
\frac{\cos(\theta/2)-\cos((q+\tfrac12)\theta)}
     {\sin(\theta/2)}
\]
gives, at $\theta_j=(2j-1)\pi/d$,
\[
\cot(\theta_j/2)=
\begin{cases}
\displaystyle
\sigma_j+2\sum_{r=1}^{m-1}\sin(r\theta_j), & d=2m,\\[6pt]
\displaystyle
2\sum_{r=1}^{m}\sin(r\theta_j), & d=2m+1.
\end{cases}
\]
Subtracting this identity from \eqref{nodeS}, we obtain
\[
\sigma_jF(\theta_j)-\cot(\theta_j/2)=
\begin{cases}
\displaystyle
\sigma_j(a_m-1)
+2\sum_{r=1}^{m-1}(a_r-1)\sin(r\theta_j), & d=2m,\\[6pt]
\displaystyle
2\sum_{r=1}^{m}(a_r-1)\sin(r\theta_j), & d=2m+1.
\end{cases}
\]
Since $0<\theta_j<\pi$, we may divide by
$\cot(\theta_j/2)>0$. The inequalities
\[
|\sin(r\theta_j)|\leq r\sin\theta_j,
\qquad
2\sin\theta_j\tan(\theta_j/2)
=4\sin^2(\theta_j/2)\leq4
\]
then give
\[
\frac{2|a_r-1|\,|\sin(r\theta_j)|}{\cot(\theta_j/2)}
\leq4r|a_r-1|.
\]
For $d=2m$, the central term additionally satisfies
\[
\frac{|a_m-1|}{\cot(\theta_j/2)}
\leq
\cot\left(\frac{\pi}{4m}\right)|a_m-1|,
\]
because $\theta_j\leq\pi-\pi/(2m)$. Consequently,
\[
\left|
\frac{\sigma_jF(\theta_j)}{\cot(\theta_j/2)}-1
\right|
\leq
\begin{cases}
\displaystyle
4\sum_{r=1}^{m-1}r|a_r-1|
+\cot\left(\frac{\pi}{4m}\right)|a_m-1|, & d=2m,\\[6pt]
\displaystyle
4\sum_{r=1}^{m}r|a_r-1|, & d=2m+1.
\end{cases}
\]
Thus, if the corresponding bound is strictly less than $1$,
then $\sigma_jF(\theta_j)>0$ at every node.
\end{proof}

We now derive conditions which determine the location of the remaining two zeros of $H(y)$.
\begin{lemma}\label{endpoints}
Let
\[
H(y)=A(1+y^d)+\sum_{r=1}^{d-1}a_ry^r,
\qquad A>0,\quad a_r=a_{d-r}>0,
\]
and suppose that $\sigma_jF(\theta_j)>0$ at $\theta = \theta_j$ as defined in
Lemma~\ref{angular}. Write $d=2m+\epsilon$,
$\epsilon\in\{0,1\}$, and set
\[
E=
\begin{cases}
H(-1),&\epsilon=0,\\
H'(-1),&\epsilon=1.
\end{cases}
\]
Then:
\begin{enumerate}[(i)]
\item If $E>0$, all zeros of $H$ lie on $\Sone$ and
$\operatorname{ord}_{-1}H=\epsilon$.
\item If $E=0$, all zeros of $H$ lie on $\Sone$ and
$\operatorname{ord}_{-1}H=2+\epsilon$.
\item If $E<0$, the zeros of $H$ off $\Sone$ are
$-\rho$ and $-\rho^{-1}$ for some $\rho>1$, both simple.
\end{enumerate}
\end{lemma}

\begin{proof}
By reciprocity,
\[
H(y)=(y+1)^\epsilon K(y),\qquad
K(y)=y^{2m}K(y^{-1}),\qquad K(0)=A,\qquad K(-1)=E.
\]
Put $G(\theta)=e^{-im\theta}K(e^{i\theta})$. Then
\[
F(\theta)=\bigl(2\cos(\theta/2)\bigr)^\epsilon G(\theta),
\]
so that the condition $\sigma_j F(\theta_j) > 0$ implies
\[
(-1)^{j-1}G(\theta_j)>0\quad(1\leq j\leq m),
\qquad G(\pi)=(-1)^mE.
\]
Thus, the intermediate value theorem ascertains a zero in each of the $m-1$ intervals $(\theta_j,\theta_{j+1})$, with their conjugates in total giving $2m-2$ distinct zeros of $K$ on
$\Sone\setminus\{-1\}$.\\

If $E>0$, then $G(\theta_m)G(\pi)<0$. A zero in
$(\theta_m,\pi)$ and its conjugate account for the remaining
two zeros of $K$. If $E=0$, differentiating the reciprocal identity at $y=-1$
gives
\[
K'(-1)=-mK(-1)=0.
\]
Hence $(y+1)^2\mid K$. Together with the $2m-2$ zeros already
obtained, this gives $\operatorname{ord}_{-1}K=2$ and all
zeros of $K$ on $\Sone$. If $E<0$, then $K(-1)<0<K(0)$, so $K$ has a zero
$\alpha\in(-1,0)$. Reciprocity gives the distinct zero
$\alpha^{-1}<-1$. These and the $2m-2$ unit-circle zeros
exhaust $\deg K=2m$, so both are simple. Since $H(y)=(y+1)^\epsilon K(y)$, the lemma is proved.
\end{proof}

The following lemma bounds the coefficient sums in
Lemma~\ref{angular}.

\begin{lemma}\label{handfold}
Let $d\geq2$, $m=\lfloor d/2\rfloor$, and $1\leq j\leq m$.
Suppose
\[
a_r=v_rv_{d-r},
\qquad
|v_r-1|\leq e_r:=c4^{-r}\quad(r\geq j),
\]
where $c>0$. Put
\[
E_j=\sum_{r=j}^{\infty}re_r
=c4^{-j}\left(\frac{4j}{3}+\frac49\right).
\]
Assume
\[
e_{d-r}\leq\eta\quad(j\leq r<d/2),
\]
where $\eta\geq0$, and assume additionally $e_m<1$ if
$d=2m$. Then
\begin{equation}\label{handfoldbound}
\begin{cases}
\displaystyle
4\sum_{r=j}^{m-1}r|a_r-1|
+\cot\left(\frac{\pi}{4m}\right)|a_m-1|,
&d=2m,\\[8pt]
\displaystyle
4\sum_{r=j}^{m}r|a_r-1|,
&d=2m+1
\end{cases}
\leq4E_j(1+\eta).
\end{equation}
If $0<v_r\leq1$ for every $r\geq j$, the right-hand side
may be replaced by $4E_j$.
\end{lemma}

\begin{proof}
For $j\leq r<d/2$,
\[
|a_r-1|
\leq e_r+e_{d-r}+e_re_{d-r},
\qquad
e_re_{d-r}\leq\eta e_r.
\]

Suppose first that $d=2m+1$. Since $r\leq d-r$,
\[
\begin{aligned}
\sum_{r=j}^{m}r(e_r+e_{d-r})
&\leq
\sum_{r=j}^{m}re_r
+\sum_{r=j}^{m}(d-r)e_{d-r}\\
&=\sum_{r=j}^{d-j}re_r
\leq E_j.
\end{aligned}
\]
Also,
\[
\sum_{r=j}^{m}re_re_{d-r}
\leq\eta\sum_{r=j}^{m}re_r
\leq\eta E_j.
\]
Thus
\[
4\sum_{r=j}^{m}r|a_r-1|\leq4E_j(1+\eta).
\]

Now suppose $d=2m$. Similarly,
\[
\begin{aligned}
\sum_{r=j}^{m-1}r(e_r+e_{d-r})
&\leq
\sum_{r=j}^{m-1}re_r+\sum_{r=m+1}^{d-j}re_r\\
&\leq E_j-me_m,
\end{aligned}
\]
and
\[
\sum_{r=j}^{m-1}re_re_{d-r}\leq\eta E_j.
\]
Since $a_m=v_m^2$,
\[
|a_m-1|\leq2e_m+e_m^2.
\]
Using $\cot t<t^{-1}$ for $0<t<\pi/2$, we obtain
\[
\cot\left(\frac{\pi}{4m}\right)|a_m-1|
<
\frac{4m}{\pi}(2+e_m)e_m
<4me_m,
\]
because $e_m<1$ and $\pi>3$. Hence
\[
4\sum_{r=j}^{m-1}r|a_r-1|
+\cot\left(\frac{\pi}{4m}\right)|a_m-1|
\leq4E_j(1+\eta).
\]

Finally, if $0<v_r\leq1$, then
\[
|a_r-1|
=1-v_rv_{d-r}
\leq(1-v_r)+(1-v_{d-r})
\leq e_r+e_{d-r}.
\]
The same estimates, with the terms involving $\eta$
omitted, give the bound $4E_j$.
\end{proof}

\subsection{Elementary estimates}

For $s>2$, define
\[
D_1(s)=\sum_{\substack{n\geq3\\ n\text{ odd}}}\frac{1}{n^s},
\qquad
D_2(s)=\sum_{\substack{n\geq3\\ n\text{ odd}}} \frac{n-1}{n^s}.
\]
Comparing with the associated integral gives
\begin{align}
D_1(s)
&\leq \frac{1}{3^{s}}+\frac{1}{5^{s}}+\frac{5^{1-s}}{2(s-1)},
\label{handT}\\
D_2(s)
&\leq \frac{2}{3^{s}}+\frac{4}{5^{s}}
+\frac12\left(
\frac{5^{2-s}}{s-2}-\frac{5^{1-s}}{s-1}
\right).
\label{handD}
\end{align}
In particular,
\begin{align}
D_1(3)&<\frac1{18}, \label{handT3}\\
D_1(5)&<\frac1{200}, \label{handT5}\\
D_2(4)&<\frac1{24}, \label{handD4}\\
D_2(5)&<\frac1{91}. \label{handD5}
\end{align}
Moreover,
\[
D_1({s+2})\leq\frac19 D_1(s),\qquad
D_2({s+2})\leq\frac19 D_2(s).
\]
Since $2^sD_2(s)$ decreases with $s$, \eqref{handD4} yields
\begin{equation}\label{handDscaled}
D_2(s)<\frac23\, \cdot \,2^{-s}\qquad(s\geq4).
\end{equation}
We also use
\[
\zeta(3)>1+2^{-3}=\frac98,
\qquad
\zeta(2) = \frac{\pi^2}{6}
<\frac53.
\]

Throughout, $\chi$ is a real nonprincipal character, induced by
$\chi^*$ of conductor $f$, and $V_s(\chi)$ is defined by
\eqref{Vdef}.

\begin{lemma}\label{evenV}
Let $V_s(\chi)$ be as in \eqref{Vdef}. For every integer $s\geq4$,
\begin{equation}\label{handV}
|V_s(\chi)-1|<2^{1-s}.
\end{equation}
Furthermore, $V_3(\chi)>27/40$.
\end{lemma}

\begin{proof}
Put $x=2^{-s}\leq1/16$. At a prime $p\geq3$, every local
factor lies between
\[
1-\frac{(p-1)p^{-s}}{1-p^{-s}}
\quad\text{and}\quad
1+(p-1)p^{-s}.
\]
At $p=2$, the corresponding bounds are
$1-(x/(1-x))$ and $1+(x/(1-x))$.
Using
\[
\prod_n(1-t_n)\geq1-\sum_n t_n,
\qquad
\prod_\nu(1+t_n)\leq
\exp\left(\sum_n t_n\right),
\]
we obtain
\[
V_s(\chi)
\geq1-\frac{x}{1-x}-\frac{D_2(s)}{1-3^{-s}}
>1-\left(\frac{16}{15}+\frac{27}{40}\right)x
>1-2x
\]
and
\[
V_s(\chi)
\leq\exp\left(\frac{x}{1-x}+D_2(s)\right)
<\exp\left(\frac{26}{15}x\right).
\]
Since $e^t\leq(1-t)^{-1}$ for $0\leq t<1$,
\[
V_s(\chi)-1<
\frac{(26/15)x}{1-13/120}
=\frac{208}{107}x<2x.
\]
At $s=3$, every local factor is at least
$(1-p^{-2})/(1-p^{-3})$. Hence
\[
V_3(\chi)\geq\frac{\zeta(3)}{\zeta(2)}
>\frac{27}{40}.
\]
\end{proof}

\subsection{Estimates for odd characters}

Let $\chi$ be odd. Since
\[
V_1(\chi)=L(1,\chi^*)
\prod_{p\in T_\chi}(1-\chi^*(p)),
\]
with $L(1,\chi^*)>0$ and $\chi^*(p)\in\{\pm1\}$ for
$p\in T_\chi$, we have
\[
V_1(\chi)>0
\quad\Longleftrightarrow\quad
\chi^*(p)=-1\quad\text{for every }p\mid N\text{ with }p\nmid f.
\]

\begin{lemma}\label{oddintervals}
Suppose $V_1(\chi)>0$. For $r\geq1$, let
\[
\lambda_r=
\begin{cases}
1,&2\mid f,\\[3pt]
\displaystyle\frac{1}{1-\chi^*(2)2^{-(2r+1)}},
&2\nmid f,\quad 2\notin T_\chi,\\[8pt]
\displaystyle\frac{1-\chi^*(2)2^{-2r}}
{1-\chi^*(2)2^{-(2r+1)}},
&2\in T_\chi.
\end{cases}
\]
Set $U_r:=V_{2r+1}(\chi)/\lambda_r$, the product of the
factors at odd primes. Then
\begin{equation}\label{handU}
U_1>\frac{17}{18},
\qquad
1-\frac{9^{2-r}}{200}
<U_r<
1+\frac{9^{2-r}}{90}
\quad(r\geq2),
\end{equation}
and
\begin{equation}\label{handWgeneral}
|V_{2r+1}(\chi)-1|<\frac7{10}4^{-r}\qquad(r\geq2).
\end{equation}
More precisely:
\begin{enumerate}[(i)]
\item If $\chi^*(2)=1$, or if $2\in T_\chi$ and
$\chi^*(2)=-1$, then $V_{2r+1}(\chi)>1$ for every $r\geq1$.
\item If $2\mid f$, then
\[
V_3(\chi)>\frac{17}{18},
\qquad V_{2r+1}(\chi)>\frac{199}{200}\quad(r\geq2).
\]
\item If $\chi^*(2)=-1$ and $2\notin T_\chi$, then
\begin{equation}\label{handWnegative}
V_3(\chi)>\frac{68}{81},
\qquad
0<1-V_{2r+1}(\chi)<\frac{29}{50}4^{-r}\quad(r\geq2).
\end{equation}
\end{enumerate}
\end{lemma}

\begin{proof}
For $s=2r+1$, every local factor at an odd prime $p$ lies
between $1-p^{-s}$ and $1+(p-1)p^{-s}$. Thus
\[
U_r\geq1-D_1(2r+1),
\qquad U_r\leq e^{D_2(2r+1)}.
\]
The lower bounds in \eqref{handU} follow from
\eqref{handT3} and \eqref{handT5}. For $r\geq2$, put
$t=9^{2-r}\leq1$. By \eqref{handD5},
\[
U_r<e^{t/91}
\leq\frac1{1-t/91}
\leq1+\frac{t}{90}.
\]
Put $q=2^{-(2r+1)}$. The possibilities for $\lambda_r$ are
\[
1,\qquad
\frac1{1-q},\qquad
\frac1{1+q},\qquad
\frac{1+2q}{1+q}.
\]
For $r\geq2$,
\[
|\lambda_r-1|\leq\frac{16}{31}4^{-r},
\qquad
|U_r-1|<\frac8{45}4^{-r}.
\]
Consequently,
\[
|V_{2r+1}(\chi)-1|
<
\left(\frac{16}{31}+\frac8{45}+\frac8{1395}\right)4^{-r}
=\frac{976}{1395}4^{-r}
<\frac7{10}4^{-r}.
\]
For (i), \eqref{handT3} and geometric decay give
\[
D_1(2r+1)<\frac49q<\frac{q}{1+2q},
\qquad q\leq\frac18.
\]
Hence
\[
U_r>1-q,
\qquad
U_r>\frac{1+q}{1+2q},
\]
which prove the two assertions in (i). Part (ii) follows
from $\lambda_r=1$. For (iii), $\lambda_1=8/9$, so
\[
V_3(\chi)>\frac89\frac{17}{18}=\frac{68}{81}.
\]
For $r\geq2$, $9^{2-r}/90<q$, and therefore
\[
V_{2r+1}(\chi)<\frac{1+9^{2-r}/90}{1+q}<1.
\]
Also,
\[
1-V_{2r+1}(\chi)
<q+\frac{9^{2-r}}{200}
\leq\left(\frac12+\frac2{25}\right)4^{-r}
=\frac{29}{50}4^{-r}.
\]
\end{proof}

\medskip

\section{\bf Zeros of the polynomials \texorpdfstring{$R_k(x;\chi,\psi)$}{Rk(x;chi,psi)}}
\medskip

Put
\[
 \e=\begin{cases}1,&\psi(-1)=-1,\\2,&\psi(-1)=1,\end{cases}
 \qquad D=k-2-\e,
 \qquad P_{D,j}(X)=(X-1)^{2j}\sum_{h=0}^{D-2j}X^h.
\]
Every $P_{D,j}$ is reciprocal of degree $D$ with leading coefficient $1$. \\

For $k\geq4$ we have $D\geq0$. For $0 \leq j \leq \lfloor D/2 \rfloor $, set
\begin{equation}\label{cj}
 c_j=\frac{|B_{D+2-2j,\chi}|}{(D+2-2j)!}
      \frac{|B_{\e+2j,\psi}|}{(\e+2j)!}M^{1-\e-2j}.
\end{equation}
Note that all $c_j$ are nonnegative, and $c_j>0$ for $j\geq1$. The coefficient $c_0$ is zero exactly when $\e=1$ and $B_{1,\psi}=0$. \\

Substituting $k-s=\e+2j$ in \eqref{Rdef}, removing the $s=1$ summand which is zero, and using the last assertion of Lemma~\ref{berL} gives
\begin{equation}\label{Rdecomp}
 R_k(x;\chi,\psi)= \delta \, (x-1)^{\e} H_D(x),\qquad \text{ for } \qquad
 H_D(x)=\sum_{j=0}^{\lfloor D/2\rfloor}c_jP_{D,j}(x) \, \text{ and } \, \delta \in \{-1,1\}.
\end{equation}

We now derive a criterion for when the polynomial $H_D(x)$ has at least $D-2$ distinct zeros on the unit circle.

\begin{lemma}\label{cheb}
With the notation as above, suppose that $D\geq4$, $c_0\geq0$, and $c_j>0$ for $1\leq j\leq\lfloor D/2\rfloor$. If
\begin{equation}\label{chebtest}
 \frac94\, \cdot \, \frac{c_2}{c_1}+\sum_{j=3}^{\lfloor D/2\rfloor}4^{j-1} \, \cdot \, \frac{c_j}{c_1}<1,
\end{equation}
then $H_D$ has at least $D-2$ distinct zeros on $\Sone$.
\end{lemma}
\begin{proof} 
For $0<t<\pi$,
\[
 e^{-iDt}P_{D,j}(e^{2it})=(-4)^j \,  \sin^{2j-1}\left(t \right) \, \sin((D-2j+1)t).
\]
Let $F(t)=e^{-iDt}H_D(e^{2it})$.  At the points $t_\ell=\ell\pi/D$, $1\leq\ell\leq D-1$, we have
\begin{equation}\label{nodeR}
 (-1)^\ell F(t_\ell)=c_0+\sum_{j\geq1}(-1)^{j+1} \, 4^j \, c_j \,
 \sin^{2j-1}(t_\ell) \, \sin((2j-1)t_\ell).
\end{equation}
Here, we have used that for $j=0$, the term is $\sin((D+1)t_{\ell})/\sin t_{\ell} = \cos \ell \pi = (-1)^{\ell}$. Observe that the $j=1$ term is $4c_1 \,\sin^2t_\ell$ while the $j=2$ term is $
-16c_2\sin^3t_{\ell}\sin(3t_{\ell})$. Using \(\sin(3t_{\ell})=(\sin t_{\ell})(3-4\sin^2t_{\ell})\) and putting \(x=\sin^2 t_{\ell}\), their ratio is
\[
\frac{-16c_2\sin^3t_{\ell}\sin(3t_{\ell})}{4c_1\sin^2t_{\ell}}
=-4\frac{c_2}{c_1}x(3-4x).
\]
Thus the ratio is bounded below by
\[
 -4\frac{c_2}{c_1}\max_{0\leq x\leq1}x(3-4x)=-\frac94\frac{c_2}{c_1},
\]
as $x(3-4x)=9/16-4(x-3/8)^2$. Moreover, for $j\geq3$, the absolute value of the $j$-th
summand relative to the positive term $4c_1\sin^2t_\ell$
is at most $4^{j-1}c_j/c_1$. Hence
\[
(-1)^\ell F(t_\ell)
\geq c_0+4c_1\sin^2t_\ell
\left(
1-\frac94\frac{c_2}{c_1}
-\sum_{j=3}^{\lfloor D/2\rfloor}
4^{j-1}\frac{c_j}{c_1}
\right).
\]
Condition \eqref{chebtest} therefore gives
$(-1)^\ell F(t_\ell)>0$ for $1\leq\ell\leq D-1$.
By continuity, $F$ has a zero in each interval
$(t_\ell,t_{\ell+1})$, $1\leq\ell\leq D-2$.
Since $t\mapsto e^{2it}$ is injective on $(0,\pi)$,
these yield at least $D-2$ distinct zeros of $H_D$ on $\Sone$.

\end{proof}

We apply this lemma to the polynomial family $R_k$ below. 
\begin{proposition}\label{Rminus2}
Let $\chi$ mod $L$ and $\psi$ mod $M$ be nonprincipal real Dirichlet characters satisfying $\chi(-1)\psi(-1) = (-1)^k$. If $R_k(x;\chi,\psi)\not\equiv0$, it has at least $\deg R_k-2$ zeros on $\Sone$, counted with multiplicity.
\end{proposition}
\begin{proof}
For $r\geq1$ with $r+1\leq\lfloor D/2\rfloor$, applying the identity \eqref{Bernabs} in Lemma \ref{berL}  gives
\[
 \frac{c_{r+1}}{c_1}
 =\left(\frac{Q_\psi}{MQ_\chi}\right)^{2r}
 \frac{V_{D-2r}(\chi)}{V_D(\chi)}
 \frac{V_{\e+2+2r}(\psi)}{V_{\e+2}(\psi)}.
\]
Since $Q_\psi\mid M$, Lemma~\ref{ratios} yields
\[
 \frac{c_{r+1}}{c_1}\leq A \, \zeta(2) \,\zeta(4) \, 16^{-r},
\]
except in the conductor-$3$ case, where it gives $A \,9^{-r}$. Consequently the left side of \eqref{chebtest} is at most, respectively,
\begin{align*}
 A \, \zeta(2) \,\zeta(4) \,\left(\frac9{64}+\sum_{r\geq2}4^{-r}\right)
 &=\frac{43 \, A \, \zeta(2) \,\zeta(4) \,}{192}<\frac{1075}{1728}<1,\\
 A\left(\frac14+\sum_{r\geq2}(4/9)^r\right)
 &=\frac{109A}{180}<\frac{218}{243}<1.
\end{align*}
For these comparisons, we use the elementary integral bounds $\zeta(2)<5/3$, $\zeta(3)>9/8$, and $\zeta(4)<1+1/16+1/24=53/48<9/8$. Thus $A<40/27$ and $A \, \zeta(2) \,\zeta(4) \,<25/9$. \\

Thus, for $D\geq 4$, Lemma \ref{cheb} implies that $H_D$ has at least $D-2$ zeros on the unit-circle. For $D=2$, the bound $D-2=0$ is immediate. For $D=3$, reciprocity gives $H_D(-1)=0$. Hence the same conclusion holds for every $D\geq2$. By \eqref{Rdecomp}, this means that $R_k$ has at least $k-4$ zeros on $\Sone$ as the degree of $R_k = k-2$ with the absolute value of the leading coefficient being $\sum c_j>0$. If $D=0$ or $1$, a nonzero $H_D$ is a positive constant or a positive multiple of $1+x$, respectively.
\end{proof}

We complete the analysis of location of zeros of $R_k$ by studying its behaviour on the real axis. 
\begin{lemma}\label{Rreal}
For a nonzero polynomial $H_D$ in \eqref{Rdecomp}, there are no real zeros outside $\{1,-1\}$. The zero $-1$ occurs exactly when $D$ is odd, and is then simple. The zero $1$ occurs exactly when $c_0=0$, and is then double.
\end{lemma}
\begin{proof}
If $D$ is even, every geometric sum in $P_{D,j}(x)$ has even degree and is positive for every real $x$:
\[
 1+x+\cdots+x^{2h}=\frac{x^{2h+1}-1}{x-1}>0
\]
(with value $2h+1$ at $x=1$). Hence every summand is nonnegative and the $j=1$ summand in $H_D$ is positive for $x\ne1$ when present.\\

If $D$ is odd, factor
\[
 1+x+\cdots+x^{2h+1}=(1+x)(1+x^2+\cdots+x^{2h}).
\]
The quotient $H_D(x)/(1+x)$ is a sum of nonnegative terms, positive for $x\ne1$. It is positive at $x=-1$, so that zero is simple. The cases $D<2$ follow directly. Finally, at $x=1$, the $j=0$ term is $c_0(D+1)$. If it vanishes, $c_1>0$ and the $j=1$ term has exactly order two, while all subsequent terms have higher order.
\end{proof}
\begin{proof}[Proof of Theorem~\ref{intro:main-R-k}]
By Proposition~\ref{Rminus2}, at most two zeros can be off the circle. A nonreal off-circle zero of a real reciprocal polynomial forces the four distinct zeros $z,\bar z,z^{-1},\bar z^{-1}$. A real off-circle zero is excluded by Lemma~\ref{Rreal}. Thus all zeros are unimodular. The factor at $1$ has multiplicity $\e$ if $c_0>0$ and $\e+2$ otherwise. The zero $-1$ is present exactly when $\chi$ is odd, since $\chi(-1)\psi(-1) = (-1)^k$ gives $(-1)^D=\chi(-1)$. For completeness, the indices below $4$ are also immediate: $k=3$ gives a scalar multiple of $x-1$, of degree $1$, and $k=1,2$ give zero polynomials. This proves Conjecture \ref{intro:conjR}.
\end{proof}

\medskip 

\section{\bf Zeros of the polynomials \texorpdfstring{$S_{2n}(x;\chi,\chi)$}{S2n(x;chi,chi)}}

\medskip

Let $n\geq2$ and suppose $S_{2n}(x;\chi,\chi)\not\equiv0$.
By Lemma~\ref{berL}, there is a constant
$C_{\chi,n}\in\mathbb R\setminus\{0\}$ such that the
following reductions hold.\\

\noindent For even $\chi$, put
\begin{equation}\label{evenH}
H_d(y)=\sum_{r=0}^{d}
V_{2r+2}(\chi)V_{2d+2-2r}(\chi)y^r,
\qquad d=n-2.
\end{equation}
Then
\begin{equation}\label{Seven}
S_{2n}(x;\chi,\chi)=C_{\chi,n}xH_d(x^2).
\end{equation}
In particular,
\[
\deg S_{2n}=2n-3,
\qquad \operatorname{ord}_0S_{2n}=1.
\]
For odd $\chi$, put
\begin{equation}\label{oddH}
H_d(y)=\sum_{r=0}^{d}
V_{2r+1}(\chi)V_{2d-2r+1}(\chi)y^r,
\qquad d=n-1.
\end{equation}
Then
\begin{equation}\label{Sodd}
S_{2n}(x;\chi,\chi)=C_{\chi,n}H_d(x^2).
\end{equation}
If $V_1(\chi)>0$, then
\[
\deg S_{2n}=2n-2,
\qquad S_{2n}(0)\neq0.
\]
If $V_1(\chi)=0$, non-vanishing requires $d\geq2$, and
\[
H_d(y)=yK_{d-2}(y),\qquad K_{d-2}(0)>0.
\]
Hence
\[
\deg S_{2n}=2n-4,
\qquad \operatorname{ord}_0S_{2n}=2.
\]

\subsection{Even characters}

\begin{proposition}\label{evenzeros}
For $d\geq2$, the polynomial \eqref{evenH} satisfies
\[
\sigma_jF(\theta_j)>0\qquad(1\leq j\leq\lfloor d/2\rfloor).
\]
Consequently, it has at least $d-2$ distinct zeros on $\Sone$.
\end{proposition}

\begin{proof}
Set $v_r=V_{2r+2}(\chi)$. Lemma~\ref{evenV} gives
\[
|v_r-1|<\frac12\,4^{-r}\qquad(r\geq1).
\]
Apply Lemma~\ref{handfold} with
\[
j=1,\qquad c=\frac12,\qquad
E_1=\frac29,\qquad \eta=\frac1{32}.
\]
For $1\leq r<d/2$, we have $d-r\geq2$, and hence
\[
e_{d-r}=\frac12\,4^{-(d-r)}
\leq\frac12\,4^{-2}=\frac1{32}.
\]
Thus, we may take $\eta=1/32$ in Lemma~\ref{handfold}. This gives
\[
\left|
\frac{\sigma_jF(\theta_j)}{\cot(\theta_j/2)}-1
\right|
\leq4\cdot\frac29\left(1+\frac1{32}\right)
=\frac{11}{12}<1.
\]
The claim follows by Lemma~\ref{angular}.
\end{proof}

When $\chi$ is also given to be primitive non-principal, we have the following. 
\begin{proposition}\label{primitiveeven}
Let $\chi$ be a primitive real even non-principal character. Then, for every $n\geq2$, all nonzero zeros of $S_{2n}(x;\chi,\chi)$ lie on $\Sone$.
\end{proposition}

\begin{proof}
Since $\chi$ is primitive, $V_s(\chi)=L(s,\chi)$.
By \eqref{Seven}, it suffices to consider
\[
H_d(y)=\sum_{r=0}^{d}
L(2r+2,\chi)L(2d+2-2r,\chi)y^r,
\qquad d=n-2.
\]
For $d=0,1$, the assertion is immediate. Assume $d\geq2$.
The proof of Proposition~\ref{evenzeros} gives
$\sigma_jF(\theta_j)>0$. We verify the endpoint
inequalities of Lemma~\ref{endpoints}. The Euler product gives
\[
L(2,\chi)\geq
\prod_p(1+p^{-2})^{-1}
=\frac{\zeta(4)}{\zeta(2)}
=\frac{\pi^2}{15}>\frac{13}{20}.
\]
Also,
\[
\zeta(4)-1
<\frac1{12}.
\]
Thus, for $r\geq1$,
\[
|L(2r+2,\chi)-1|
\leq4^{1-r}(\zeta(4)-1)
<\frac13\,4^{-r}.
\]

Write
\[
H_d(y)=A(1+y^d)
+\sum_{r=1}^{d-1}(1+\delta_r)y^r.
\]
Since $d\geq2$,
\[
A=L(2,\chi)L(2d+2,\chi)
>\frac{13}{20}\left(1-\frac1{48}\right)
=\frac{611}{960}.
\]
Moreover,
\[
|\delta_r|
\leq\frac13\,4^{-r}
+\frac13\,4^{-(d-r)}
+\frac19\,4^{-d},
\qquad 1\leq r\leq d-1.
\]
Hence
\[
\sum_{r=1}^{d-1}|\delta_r|
\leq\frac23\sum_{r=1}^{\infty}4^{-r}
+\frac{d-1}{9}\,4^{-d}
\leq\frac29+\frac1{144}
=\frac{11}{48}.
\]

If $d$ is even, then
\[
H_d(-1)
\geq2A-1-\frac{11}{48}
>\frac7{160}>0.
\]
If $d$ is odd, reciprocity gives
$\delta_r=\delta_{d-r}$, and therefore
\[
\left|
\sum_{r=1}^{d-1}r\delta_r(-1)^{r-1}
\right|
\leq\frac d2\sum_{r=1}^{d-1}|\delta_r|
\leq\frac{11d}{96}.
\]
Consequently,
\[
H_d'(-1)
\geq dA-\frac{d-1}{2}-\frac{11d}{96}
>\frac{7d}{320}+\frac12>0.
\]
The claim now follows from Lemma~\ref{endpoints}.
\end{proof}

\subsection{Odd characters with
\texorpdfstring{$V_1(\chi)=0$}{V1(chi)=0}}

\begin{proposition}\label{handzeroouter}
If $V_1(\chi)=0$, every nonzero zero of \eqref{oddH} lies on
$\Sone$.
\end{proposition}

\begin{proof}
Write
\[
H_d(y)=yK_D(y),\qquad D=d-2,
\qquad
K_D(y)=\sum_{r=0}^{D}
V_{2r+3}(\chi)V_{2D+3-2r}(\chi)y^r.
\]
For $D=0,1$,
\[
K_0=V_3(\chi)^2,\qquad
K_1=V_3(\chi)V_5(\chi)(1+y),
\]
so the assertion is immediate. Suppose $D\geq2$.\\

By Lemma~\ref{evenV},
\[
V_3(\chi)>\frac{27}{40},
\qquad
|V_{2r+1}(\chi)-1|<4^{-r}\quad(r\geq2).
\]
Apply Lemma~\ref{handfold} to $v_r=V_{2r+3}(\chi)$ with
$j=1$, $c=1/4$ and $\eta=1/16$. The normalized error is
at most
\[
4\cdot\frac19\left(1+\frac1{16}\right)
=\frac{17}{36}<1.
\]
Thus $K_D$ satisfies the strict sign inequalities of Lemma~\ref{angular}, and has at least $D-2$ distinct zeros on the unit circle.\\

\noindent To analyse the location of the remaining zeros, write
\[
K_D(y)=A(1+y^D)+\sum_{r=1}^{D-1}(1+\delta_r)y^r.
\]
Since $A=V_3(\chi)V_{2D+3}(\chi)$ and $D\geq2$,
Lemma~\ref{evenV} gives
\[
A>\frac{27}{40}\bigl(1-4^{-(D+1)}\bigr)
>\frac{27}{40}\cdot\frac{15}{16}
=\frac{81}{128}.
\]
Moreover, for $1\leq r\leq D-1$, Lemma~\ref{evenV} gives
\[
|\delta_r|
=\bigl|V_{2r+3}(\chi)V_{2D+3-2r}(\chi)-1\bigr|
\leq4^{-(r+1)}+4^{-(D+1-r)}+4^{-(D+2)}.
\]
Summing and using
$\sum_{\nu=2}^{\infty}4^{-\nu}=1/12$, we obtain
\[
\sum_{r=1}^{D-1}|\delta_r|
\leq2\sum_{\nu=2}^{D}4^{-\nu}
+(D-1)4^{-(D+2)}
\leq\frac16+\frac1{256}
=\frac{131}{768},
\]
since $(D-1)4^{-(D+2)}$ decreases for $D\geq2$ and
equals $1/256$ at $D=2$.
If $D$ is even,
\[
K_D(-1)\geq2A-1-\frac{131}{768}
>\frac{73}{768}>0.
\]
If $D$ is odd, $\delta_r=\delta_{D-r}$ gives
\[
\left|
\sum_{r=1}^{D-1}r\delta_r(-1)^{r-1}
\right|
\leq\frac D2\sum_{r=1}^{D-1}|\delta_r|
\leq\frac{131D}{1536}.
\]
Therefore
\[
K_D'(-1)
\geq D\left(A-\frac12-\frac{131}{1536}\right)+\frac12
>\frac{73D}{1536}+\frac12>0.
\]
Lemma~\ref{endpoints} therefore implies that every zero of $K_D$
lies on $\Sone$.
\end{proof}

\subsection{Odd characters with
\texorpdfstring{$V_1(\chi)>0$}{V1(chi)>0}}

\begin{proposition}\label{oddzeros}
If $V_1(\chi)>0$ and $d\geq2$, the polynomial \eqref{oddH} satisfies
\[
\sigma_jF(\theta_j)>0\qquad(1\leq j\leq\lfloor d/2\rfloor).
\]
Consequently, it has at least $d-2$ distinct zeros on $\Sone$.
\end{proposition}

\begin{proof}
Write
\[
a_r=V_{2r+1}(\chi)V_{2d-2r+1}(\chi).
\]
For $d=2,3$, \eqref{nodeS} gives respectively
$F(\theta_1)=a_1>0$ and
$F(\theta_1)=2a_1\sin(\pi/3)>0$.
Assume $d\geq4$.\\

Fix $1\leq j\leq\lfloor d/2\rfloor$ and put
$x=\sin^2(\theta_j/2)$. Subtracting the cotangent identity
from \eqref{nodeS} gives
\[
\frac{\sigma_jF(\theta_j)}{\cot(\theta_j/2)}
=1+4x(a_1-1)+R_j,
\]
where $R_j$ contains the terms involving $a_r-1$ for
$2\leq r\leq\lfloor d/2\rfloor$.
By Lemma~\ref{oddintervals},
\[
|V_{2r+1}(\chi)-1|<\frac7{10}4^{-r}\qquad(r\geq2).
\]
Applying Lemma~\ref{handfold} with initial index $2$,
$c=7/10$ and $\eta=7/160$, we obtain
\begin{equation}\label{handoddrest}
|R_j|
\leq4\cdot\frac{49}{360}
\left(1+\frac7{160}\right)
=\frac{8183}{14400}.
\end{equation}
Here $\eta=7/160$ is valid since $d-r\geq2$ for
$2\leq r<d/2$.\\

In case~(i) of Lemma~\ref{oddintervals}, $a_1>1$, and hence
\[
\frac{\sigma_jF(\theta_j)}{\cot(\theta_j/2)}
\geq1-\frac{8183}{14400}>0.
\]
In case~(ii),
\[
a_1>\frac{17}{18} \, \cdot \, \frac{199}{200}
=\frac{3383}{3600}.
\]
Since $0<x<1$, it follows that
\[
\frac{\sigma_jF(\theta_j)}{\cot(\theta_j/2)}
>1-4\left(1-\frac{3383}{3600}\right)
-\frac{8183}{14400}
=1-\frac{259}{320}>0.
\]
It remains to consider case~(iii). If $d=4$, then
\[
a_1=V_3(\chi)V_7(\chi)>
\frac{68}{81} \, \cdot \,\frac{99}{100}>\frac45,
\qquad
0<a_2=V_5(\chi)^2<1.
\]
Equation \eqref{nodeS} gives
\[
F(\theta_1)=a_2+\sqrt2\,a_1>0,
\qquad
-F(\theta_2)=-a_2+\sqrt2\,a_1
>-1+\frac75 \, \cdot \,\frac45>0.
\]

Suppose $d\geq5$. By \eqref{handWnegative},
\[
a_1>\frac{68}{81}\, \cdot \, \frac{399}{400}>\frac56,
\qquad
0<1-a_2<
\frac{29}{50}\left(\frac1{16}+\frac1{64}\right)
<\frac1{20}.
\]
Separating the terms with $r=1,2$, we have
\[
\frac{\sigma_jF(\theta_j)}{\cot(\theta_j/2)}
=
1+4x(a_1-1)+8x(1-2x)(a_2-1)+R_j',
\]
where we used
\[
2\sin\theta_j\tan(\theta_j/2)=4x,
\qquad
2\sin(2\theta_j)\tan(\theta_j/2)=8x(1-2x).
\]
For $d=5$, $R_j'=0$. For $d\geq6$, the inequalities
\[
0<V_{2r+1}(\chi)<1,
\qquad
1-V_{2r+1}(\chi)<\frac{29}{50}4^{-r}
\quad(r\geq2)
\]
allow us to apply the last assertion of
Lemma~\ref{handfold} with initial index $3$. Thus
\[
|R_j'|
\leq4\cdot\frac{29}{50}
\sum_{r=3}^{\infty}r4^{-r}
=\frac{29}{180}<\frac16.
\]
If $x\geq1/2$, then
$8x(1-2x)(a_2-1)\geq0$, so
\[
\frac{\sigma_jF(\theta_j)}{\cot(\theta_j/2)}
>1-\frac23-\frac16=\frac16>0.
\]
If $0<x\leq1/2$, then $4x\leq2$ and
$0\leq8x(1-2x)\leq1$, giving
\[
\frac{\sigma_jF(\theta_j)}{\cot(\theta_j/2)}
>1-\frac13-\frac1{20}-\frac16
=\frac9{20}>0.
\]
Since $\cot(\theta_j/2)>0$, we have
$\sigma_jF(\theta_j)>0$ in every case.
Therefore, the claim is proved by Lemma~\ref{angular}.
\end{proof}

\subsection{Zeros of \texorpdfstring{$S_{2n}$}{S2n}}

\begin{theorem}\label{Szerolocation}
Let $\chi$ be real and nonprincipal, and suppose
$S_{2n}(x;\chi,\chi)\not\equiv0$. Then at most four
nonzero zeros lie off $\Sone$, counted with multiplicity.
If such zeros occur, they are simple and have the form
\[
\pm i\sqrt{\rho},\qquad
\pm i/\sqrt{\rho},
\qquad \rho>1.
\]
There are no nonzero real zeros.\\

If $\chi$ is odd and $V_1(\chi)=0$, every nonzero zero lies
on $\Sone$. Otherwise, for the reduced polynomial $H_d$
with $d\geq2$, put
\[
E=
\begin{cases}
H_d(-1),&d\text{ even},\\
H_d'(-1),&d\text{ odd}.
\end{cases}
\]
Then all nonzero zeros of $S_{2n}$ lie on $\Sone$ if and
only if $E\geq0$.
\end{theorem}

\begin{proof}
If the reduced degree is $0$ or $1$, reciprocity gives
a nonzero constant or a positive multiple of $1+y$.
For odd $\chi$, the case $V_1(\chi)=0$ follows from
Proposition~\ref{handzeroouter}.\\

In the remaining cases, Propositions~\ref{evenzeros}
and \ref{oddzeros}, together with Lemma~\ref{endpoints},
show that the only possible zeros of $H_d$ off $\Sone$
are a simple pair $-\rho,-\rho^{-1}$, occurring precisely
when $E<0$. Under $y=X^2$, these give
\[
X=\pm i\sqrt{\rho},\qquad X=\pm i/\sqrt{\rho}.
\]
Finally, the reduced polynomials have positive
coefficients, so they have no positive real zeros.
Equations \eqref{Seven} and \eqref{Sodd} therefore
exclude nonzero real zeros of $S_{2n}$.
\end{proof}

\begin{proof}[Proof of Theorem~\ref{intro:main-S-2n}]
For $n\geq2$, Theorem~\ref{Szerolocation} shows that
$S_{2n}(x;\chi,\chi)$ has at most four nonzero zeros
off $\Sone$. If such zeros occur, they are
\[
\pm i\sqrt{\rho},\qquad
\pm i/\sqrt{\rho},
\qquad \rho>1.
\]
For primitive real even $\chi$, Proposition~\ref{primitiveeven}
shows that every nonzero zero lies on $\Sone$. For $n=1$, the definition gives
$S_2(x;\chi,\chi)=B_{1,\chi}^2$, which is constant.
This completes the proof. 
\end{proof}


\medskip
\section{\bf The function \texorpdfstring{$\Phi_{\chi}(t)$}{Phi chi(t)} and its properties}

\medskip

Throughout this section and the next, $\chi$ is a real nonprincipal character
which is either even and imprimitive or odd with $B_{1,\chi}\neq0$.
We first study the arithmetic properties of $\Phi_\chi$ and
$I_\chi$. The cotangent expansion expresses $I_\chi/\pi$ as a
rational function of $e^{2\pi/Q}$ with algebraic coefficients;
Nesterenko's theorem then gives algebraic independence of
$I_\chi$ and $\pi$. We then use the endpoint estimates and local uniform
convergence to determine the zeros of $\Phi_\chi$ in the unit disk.

\begin{proof}[Proof of Theorem~\ref{arithmeticlimit}]
Put
\[
\alpha_r=
L(2r+\e,\chi^*)
\prod_{\substack{p\mid N\\p\nmid f}}
\bigl(1-\chi^*(p)p^{1-2r-\e}\bigr),
\qquad
\delta_r=\alpha_r-1.
\]
By Lemmas~\ref{evenV} and \ref{oddintervals}, we have
$\delta_r=O_\chi(4^{-r})$.
Thus $\sum_{r\geq0}\delta_r(-t)^r$ converges absolutely
for $|t|<4$, and
\[
\Phi_\chi(t)=\frac1{1+t}+\sum_{r=0}^{\infty}\delta_r(-t)^r
\]
is meromorphic in this disk, with a unique pole at $t=-1$,
which is simple and has residue $1$. In particular,
\[
I_\chi=\frac12+\sum_{r=0}^{\infty}(-1)^r\delta_r
\]
converges absolutely and is real. \\

We first prove that $I_\chi$ and $\pi$ are algebraically independent. We do not assume that $I_{\chi} \neq 0$. Set
$$
Q=f\prod_{\substack{p\mid N\\p\nmid f}}p,
\qquad
b(\nu)=\chi^*(\nu)
\prod_{\substack{p\mid N,\ p\nmid f\\p\mid\nu}}(1-p).
$$
Then $b$ is periodic modulo $Q$ and satisfies
$$
b(-\nu)=(-1)^\e b(\nu),
\qquad b(1)=1.
$$
Expanding the finite Euler product, we obtain
$$
L(s,\chi^*)
\prod_{\substack{p\mid N\\p\nmid f}}
\bigl(1-\chi^*(p)p^{1-s}\bigr)
=
\sum_{\nu=1}^{\infty}\frac{b(\nu)}{\nu^s}
\qquad(s>1).
$$
When $\e=1$, this identity also holds at $s=1$. Indeed,
$\chi^*$ is then odd, hence nonprincipal, and its Dirichlet
series converges at $s=1$. Expanding the product defining
$b(\nu)$, we obtain
$$
\sum_{\nu\le X}\frac{b(\nu)}{\nu}
=
\sum_{d\mid Q/f}\mu(d)\chi^*(d)
\sum_{m\le X/d}\frac{\chi^*(m)}m.
$$
Letting $X\to\infty$ therefore yields
$$
\sum_{\nu=1}^{\infty}\frac{b(\nu)}{\nu}
=
L(1,\chi^*)
\prod_{\substack{p\mid N\\p\nmid f}}
\bigl(1-\chi^*(p)\bigr).
$$
Since $b(1)=1$, it follows that
$$
\alpha_r-1
=
\sum_{\nu=2}^{\infty}
\frac{b(\nu)}{\nu^{2r+\e}}
\qquad(r\ge0).
$$
We separate the term $r=0$, whose series may converge only
conditionally when $\e=1$. Since $b$ is bounded,
$$
\sum_{r=1}^{\infty}\sum_{\nu=2}^{\infty}
\frac{|b(\nu)|}{\nu^{2r+\e}}
=
\sum_{\nu=2}^{\infty}
\frac{|b(\nu)|}{\nu^\e(\nu^2-1)}
<\infty.
$$
Thus the remaining sums may be interchanged. Using
$$
\sum_{r=1}^{\infty}(-1)^r\nu^{-2r}
=-\frac1{\nu^2+1},
$$
we obtain
$$
\begin{aligned}
I_\chi
&=\frac12+
  \sum_{\nu=2}^{\infty}\frac{b(\nu)}{\nu^\e}
  +\sum_{\nu=2}^{\infty}\frac{b(\nu)}{\nu^\e}
    \sum_{r=1}^{\infty}(-1)^r\nu^{-2r}\\
&=\frac12+
  \sum_{\nu=2}^{\infty}
  \frac{b(\nu)\nu^{2-\e}}{\nu^2+1}\\
&=\sum_{\nu=1}^{\infty}
  \frac{b(\nu)\nu^{2-\e}}{\nu^2+1}.
\end{aligned}
$$

Now, put $\vartheta=2\pi/Q$. We apply the
cotangent expansion
$$
\lim_{M\to\infty}\sum_{k=-M}^{M}\frac1{z+kQ}
=
\frac{\pi}{Q}\cot\left(\frac{\pi z}{Q}\right)
\qquad(z\notin Q\mathbb Z)
$$
at $z=a+i$. Taking real and imaginary parts gives
$$
\begin{aligned}
\lim_{M\to\infty}\sum_{k=-M}^{M}
\frac{a+kQ}{(a+kQ)^2+1}
&=
\frac{\pi}{Q}
\frac{\sin(a\vartheta)}{\cosh\vartheta-\cos(a\vartheta)},\\
\sum_{k\in\mathbb Z}\frac1{(a+kQ)^2+1}
&=
\frac{\pi}{Q}
\frac{\sinh\vartheta}{\cosh\vartheta-\cos(a\vartheta)}.
\end{aligned}
$$
Here we have used
$$
\cot(u+iv)
=
\frac{\sin(2u)-i\sinh(2v)}
     {\cosh(2v)-\cos(2u)}.
$$
Using the parity of $b$ and the fact that $b(0)=0$, we have
$$
2I_\chi
=
\lim_{M\to\infty}
\sum_{\nu=-MQ}^{MQ}
\frac{b(\nu)\nu^{2-\e}}{\nu^2+1}.
$$
Since $b(Q)=b(0)=0$, the multiples of $Q$ contribute nothing to the above sum.
For $1\le a\le Q-1$, the condition
$$
-MQ\le a+kQ\le MQ
$$
is equivalent to $-M\le k\le M-1$. Hence
$$
\sum_{\nu=-MQ}^{MQ}
\frac{b(\nu)\nu^{2-\e}}{\nu^2+1}
=
\sum_{a=1}^{Q-1}b(a)
\sum_{k=-M}^{M-1}
\frac{(a+kQ)^{2-\e}}{(a+kQ)^2+1}.
$$
Extending each inner sum to $k=M$ adds just the term
$$
b(a)\frac{(a+MQ)^{2-\e}}{(a+MQ)^2+1}.
$$
Since $Q$ is fixed and $b$ is bounded, the sum of the
absolute values of these added terms is
$O_\chi(M^{-\e})$, and therefore tends to zero.
Thus we may use the symmetric bounds $-M\le k\le M$
when taking the limit. Grouping by residue classes and using
the preceding identities, we derive
$$
I_\chi
=
\frac{\pi}{2Q}
\begin{cases}
\displaystyle
\sum_{a=1}^{Q}
\frac{b(a)\sin(a\vartheta)}{\cosh\vartheta-\cos(a\vartheta)},
&\e=1,\\[10pt]
\displaystyle
\sinh\vartheta\sum_{a=1}^{Q}
\frac{b(a)}{\cosh\vartheta-\cos(a\vartheta)},
&\e=2.
\end{cases}
$$
Therefore, we write $I_\chi/\pi=\mathcal R_\chi(e^\vartheta)$,
where
$$
\mathcal R_\chi(Z)=
\begin{cases}
\displaystyle
\frac1Q\sum_{a=1}^{Q}
\frac{b(a)Z\sin(a\vartheta)}
     {Z^2-2Z\cos(a\vartheta)+1},
&\e=1,\\[12pt]
\displaystyle
\frac1{2Q}\sum_{a=1}^{Q}
\frac{b(a)(Z^2-1)}
     {Z^2-2Z\cos(a\vartheta)+1},
&\e=2.
\end{cases}
$$
This rational function has algebraic coefficients, since $b(a)$, $\sin(a\vartheta)$
and $\cos(a\vartheta)$ are algebraic. Its denominators do not
vanish at $Z=e^\vartheta>1$.\\

First, we see that $\mathcal R_\chi$ is nonconstant. The denominator of the $a$-th term factors as
$$
Z^2-2Z\cos(a\vartheta)+1=(Z-e^{ia\vartheta})(Z-e^{-ia\vartheta}).
$$
Since $\vartheta=2\pi/Q$, this denominator vanishes at $Z=e^{i\vartheta}$
precisely when $a\equiv\pm1\pmod Q$. Thus, among
$1\le a\le Q$, only the terms with $a=1$ and $a=Q-1$
can have a pole there. These indices are distinct because
$Q>2$. Using $b(1)=1$, $b(Q-1)=(-1)^\e$ and
$\sin((Q-1)\vartheta)=-\sin\vartheta$, the sum of these two terms is
$$
\begin{cases}
\displaystyle
\frac{2Z\sin\vartheta}
     {Q(Z^2-2Z\cos\vartheta+1)},
&\e=1,\\[12pt]
\displaystyle
\frac{Z^2-1}
     {Q(Z^2-2Z\cos\vartheta+1)},
&\e=2.
\end{cases}
$$
Now $Q>2$ implies $0<\vartheta<\pi$, so $\sin\vartheta\neq0$ and
$e^{2i\vartheta}\neq1$. Hence, in either case, the numerator is
nonzero at $Z=e^{i\vartheta}$, while the denominator has a simple
zero there. All the remaining terms are holomorphic at
this point and therefore cannot cancel this pole.
Thus, the rational function is nonconstant. \\

By Nesterenko's theorem \cite{Nesterenko}, the numbers
$\pi$ and $e^\pi$ are algebraically independent. The same is therefore true of $\pi$ and $e^\vartheta$. 
Clearing denominators in
$\mathcal R_\chi(e^\vartheta)=I_\chi/\pi$, we obtain a
polynomial relation between $e^\vartheta$ and $I_\chi/\pi$.
This relation is nonzero,
even if $I_\chi=0$, since otherwise the rational function would be identically equal to the constant $I_\chi/\pi$. It follows that $e^\vartheta$ is algebraic over
$\overline{\mathbb Q}(I_\chi/\pi)$.
If $I_\chi$ were algebraic over $\overline{\mathbb Q}(\pi)$,
then so would be $I_\chi/\pi$, and hence also $e^\vartheta$.
This contradicts the algebraic independence of $\pi$
and $e^\vartheta$. Therefore $I_\chi$ and $\pi$ are algebraically
independent. In particular, $I_\chi$ is transcendental
and thus nonzero. This proves (i).\\

We next determine the zeros of $\Phi_\chi$ in $|t|<1$.
By \eqref{Seven} and \eqref{Sodd}, with $d=n-\e$,
\[
S_{2n}(x;\chi,\chi)
=C_{\chi,n}x^{\e-1}H_d(x^2),
\qquad
H_d(y)=\sum_{r=0}^{d}\alpha_r\alpha_{d-r}y^r,
\]
where $C_{\chi,n}\neq0$ and $\alpha_r>0$ for every $r\geq0$.
We determine the eventual signs of $H_d(-1)$ and
$H_d'(-1)$ by expanding $\alpha_r\alpha_{d-r}$.\\

For even $d$, symmetry gives
\[
\begin{aligned}
H_d(-1)
=1+2\sum_{r=0}^{d}(-1)^r\delta_r
  +\sum_{r=0}^{d}(-1)^r\delta_r\delta_{d-r}=2I_\chi+O_\chi\left( \frac{d}{4^{d}} \right).
\end{aligned}
\]
For odd $d$, the same expansion gives
\[
\begin{aligned}
\frac{H_d'(-1)}{d}=\frac{d+1}{2d}
 +\sum_{r=0}^{d}
   \left(1-\frac{2r}{d}\right)(-1)^r\delta_r
 +\frac1d\sum_{r=0}^{d}
   r(-1)^{r-1}\delta_r\delta_{d-r}=I_\chi+O_\chi\left( \frac{1}{d} \right),
\end{aligned}
\]
since $\sum_{r\geq0}r|\delta_r|<\infty$. Since $I_\chi\neq0$, both endpoint expressions have its sign
for all sufficiently large $d$. Propositions
\ref{evenzeros} and \ref{oddzeros}, together with
Lemma~\ref{endpoints}, show that $H_d$ has no zeros off
$\Sone$ when $I_\chi>0$, and exactly one zero in $|y|<1$
when $I_\chi<0$, for all sufficiently large $d$. \\

For $|t|<1$, we have
\[
\Phi_\chi(t)=\sum_{r=0}^{\infty}\alpha_r(-t)^r.
\]
Choose constants $C,M>0$ such that
$|\alpha_r-1|\leq C4^{-r}$ and $|\alpha_r|\leq M$ for all $r\geq0$.
For $1/4<R<1$,
\[
\begin{aligned}
\sup_{|t|\leq R}|H_d(-t)-\Phi_\chi(t)|
&\leq MC\sum_{r=0}^d4^{-(d-r)}R^r
   +M\sum_{r>d}R^r\\
&\leq \left(\frac{MC}{1-(4R)^{-1}}+
                  \frac{MR}{1-R}\right)R^d.
\end{aligned}
\]
Thus $H_d(-t)\to\Phi_\chi(t)$ locally uniformly on $|t|<1$. \\

If $I_\chi>0$, Hurwitz's theorem
shows that $\Phi_\chi$ has no zeros in $|t|<1$, since
$\Phi_\chi(0)=\alpha_0>0$. This proves the assertion (ii). \\

If $I_\chi<0$, then
\[
\Phi_\chi(0)=\alpha_0>0,\qquad \Phi_\chi(1)=I_\chi<0,
\]
so $\Phi_\chi$ has a zero $t_\chi\in(0,1)$.
\noindent Choose $\delta>0$ such that the closed disk
$|t-t_\chi|\leq\delta$ lies in $|t|<1$ and contains
no zero of $\Phi_\chi$ other than $t_\chi$.
Since $\Phi_\chi$ has no zeros on its boundary,
\[
m_\delta:=
\min_{|t-t_\chi|=\delta}|\Phi_\chi(t)|>0.
\]
Local uniform convergence gives, for all sufficiently
large $d$,
\[
\max_{|t-t_\chi|=\delta}
|H_d(-t)-\Phi_\chi(t)|<m_\delta.
\]
Rouch\'e's theorem therefore implies that $H_d(-t)$
and $\Phi_\chi(t)$ have the same number of zeros
in this disk, counted with multiplicity. By Lemma~\ref{endpoints}, $H_d(-t)$ has exactly one
zero in $|t|<1$ for all sufficiently large $d$.
It follows that $t_\chi$ is simple. If $\Phi_\chi$
had another zero in the unit disk, applying the same
argument on a disjoint disk around that zero would
give two zeros of $H_d(-t)$ in $|t|<1$, a contradiction.
Thus $t_\chi$ is unique.\\

It remains to prove the arithmetic assertion about $t_\chi$
when $I_\chi<0$.
For $0<t\leq1$, interchanging sums as in the argument about $I_{\chi}$, we obtain
\[
\Phi_\chi(t)=\sum_{\nu=1}^{\infty}
\frac{b(\nu)\nu^{2-\e}}{\nu^2+t}.
\]
Indeed, the $r\geq1$ terms are absolutely summable uniformly
for $0<t\leq1$, and
$\sum_{r\geq1}(-t/\nu^2)^r=-t/(\nu^2+t)$ for $\nu\geq2$.
We now apply the cotangent expansion at $a+i\sqrt t$ instead of
$a+i$ and obtain
\[
\Phi_\chi(t)=
\begin{cases}
\pi\mathcal R_\chi(e^{\vartheta\sqrt t}),&\e=1,\\[3pt]
\displaystyle\frac{\pi}{\sqrt t}
\mathcal R_\chi(e^{\vartheta\sqrt t}),&\e=2.
\end{cases}
\]
The denominators of $\mathcal R_\chi$ have all their zeros
on $\Sone$. Thus
$\mathcal R_\chi(e^{\vartheta\sqrt{t_\chi}})=0$ implies that
\[
\beta_\chi:=e^{\vartheta\sqrt{t_\chi}}
\]
is algebraic, since $\mathcal R_\chi$ is a nonzero rational
function with algebraic coefficients.
If $t_\chi$ were algebraic, then
\[
\beta_\chi=(-1)^{-2i\sqrt{t_\chi}/Q},
\]
with $\log(-1)=i\pi$, would be transcendental by the
Gelfond--Schneider theorem \cite[Chapter~X]{Niven}.
This contradiction proves that $t_\chi$ is transcendental,
and hence so is $\rho_\chi=t_\chi^{-1}$.
\end{proof}

For primitive odd characters, the infinite expression of $I_{\chi}$ above
can be written in terms of a value of a Fekete polynomial. This identifies
both the sign of $I_\chi$ and the algebraic number determining
$t_\chi$. The distribution of zeros of Fekete polynomials associated with quadratic characters of prime conductor was studied by Conrey, Granville, Poonen and Soundararajan \cite{CGPS}.

\begin{proposition}\label{Feketerelation}
Let $\chi$ be a primitive odd real character of conductor
$f$, and define
\[
P_\chi(z)=\sum_{a=1}^{f-1}\chi(a)z^a.
\]
For $0<t\leq1$,
\[
\Phi_\chi(t)=
\frac{\pi}{\sqrt f}\,
\frac{P_\chi(e^{-2\pi\sqrt t/f})}
     {1-e^{-2\pi\sqrt t}}.
\]
Consequently,
\[
I_\chi=
\frac{\pi}{\sqrt f}\, \left( 
\frac{P_\chi(e^{-2\pi/f})}{1-e^{-2\pi}} \right).
\]
If $I_\chi<0$, then $P_\chi$ has exactly one zero in $(e^{-2\pi/f},1)$, say $q_{\chi}$. Then $q_{\chi}$ is simple, and
\[
q_\chi=e^{-2\pi/(f\sqrt{\rho_\chi})}.
\]
\end{proposition}

To prove the proposition, we expand the trigonometric kernel
and use the finite Fourier transform of a primitive character.

\begin{proof}[Proof of Proposition~\ref{Feketerelation}]
Since $\chi$ is primitive and odd, $Q=f$ and $b(a)=\chi(a)$.
From the cotangent formula in the preceding proof, we obtain,
for $0<t\leq1$,
\[
\Phi_\chi(t)=\frac{\pi}{2f}
\sum_{a=1}^{f}
\frac{\chi(a)\sin(2\pi a/f)}
     {\cosh(2\pi\sqrt t/f)-\cos(2\pi a/f)}.
\]
For $u>0$,
\[
\frac{\sin\theta}{\cosh u-\cos\theta}
=2\sum_{m=1}^{\infty}e^{-mu}\sin(m\theta).
\]
The Gauss-sum identity for a primitive odd real character gives
\[
\sum_{a=1}^{f}\chi(a)e^{2\pi i ma/f}
=i\sqrt f\,\chi(m),
\qquad
\sum_{a=1}^{f}\chi(a)\sin(2\pi ma/f)
=\sqrt f\,\chi(m).
\]
Putting $q=e^{-2\pi\sqrt t/f}$, we therefore obtain
\[
\Phi_\chi(t)=\frac{\pi}{\sqrt f}\left(
\sum_{m=1}^{\infty}\chi(m)q^m \right)
=\frac{\pi}{\sqrt f} \left( \frac{P_\chi(q)}{1-q^f} \right),
\]
where the last equality follows by periodicity. Taking $t=1$
gives the formula for $I_\chi$.\\

If $I_\chi<0$, Theorem~\ref{arithmeticlimit} gives a unique
simple zero $t_\chi\in(0,1)$. The map
$t\mapsto e^{-2\pi\sqrt t/f}$ is a bijection from $(0,1)$
onto $(e^{-2\pi/f},1)$ with nonzero derivative. The positive
factor $\pi/(\sqrt f(1-q^f))$ does not vanish there.
Thus $P_\chi$ has exactly one zero in this interval, it is
simple, and
\[
q_\chi=e^{-2\pi\sqrt{t_\chi}/f}
=e^{-2\pi/(f\sqrt{\rho_\chi})}.
\]
This completes the proof. 
\end{proof}

\medskip 

\section{\bf A criterion for eventual occurrence of exceptional zeros}

\medskip

We now apply the endpoint estimates from the proof of
Theorem~\ref{arithmeticlimit} to determine the eventual
zeros of $S_{2n}$. We then use local uniform convergence
to identify the limits of the exceptional zeros and
estimate their rate of convergence.

\begin{proof}[Proof of Theorem~\ref{eventualzeros}]
Under the hypotheses of this theorem, we have $\alpha_0>0$. By
\eqref{Seven} and \eqref{Sodd}, with $d=n-\e$,
\[
S_{2n}(x;\chi,\chi)=C_{\chi,n}x^{\e-1}H_d(x^2),
\qquad
H_d(y)=\sum_{r=0}^{d}\alpha_r\alpha_{d-r}y^r,
\]
where $C_{\chi,n}\neq0$. The endpoint estimates in the proof
of Theorem~\ref{arithmeticlimit} give
\[
H_d(-1)=2I_\chi+O_\chi \left( \frac{d}{4^{d}} \right)\quad(d\text{ even}),
\qquad
\frac{H_d'(-1)}d=I_\chi+O_\chi\left( \frac{1}{d} \right)\quad(d\text{ odd}).
\]
Since $I_\chi\neq0$, Lemma~\ref{endpoints} and Propositions
\ref{evenzeros} and \ref{oddzeros} show that, for every
sufficiently large $d$, all zeros of $H_d$ lie on $\Sone$
if $I_\chi>0$. If $I_\chi<0$, its only zeros off $\Sone$
are a simple negative reciprocal pair
$-\rho_n,-\rho_n^{-1}$ with $\rho_n>1$.
The substitution $y=x^2$ proves the asserted alternatives
for the zeros of $S_{2n}$.\\

Suppose $I_\chi<0$. Then, the Rouch\'e argument in the proof of
Theorem~\ref{arithmeticlimit} applies to every sufficiently
small disk $|t-t_\chi|<\delta$.
The zero of $H_d(-t)$ in the disk about $t_\chi$
is $\rho_n^{-1}$. Since $\delta$ can be chosen
arbitrarily small, we obtain
\[
\rho_n^{-1}\longrightarrow t_\chi,
\qquad
\rho_n\longrightarrow t_\chi^{-1}>1.
\]
Finally, fix $R$ with $\max\{1/4,t_\chi\}<R<1$. Since $t_\chi$ is simple,
there are $c,\delta>0$ such that
$|\Phi_\chi(t)|\geq c|t-t_\chi|$ when $|t-t_\chi|\leq\delta$.
For all sufficiently large $n$, the uniform estimate in the proof
of Theorem~\ref{arithmeticlimit}, applied at $t=\rho_n^{-1}$, gives
\[
|\rho_n^{-1}-t_\chi|=O_{\chi,R}(R^d).
\]
Since $d=n-\e$ and $\rho_n\to t_\chi^{-1}$, this proves
$\rho_n-\rho_\chi=O_{\chi,R}(R^n)$.

\end{proof}

\medskip

\section{\bf Concluding remarks}

\medskip

The counterexample \eqref{intro:example} shows that the
Riemann hypothesis for period polynomials does not extend
to all Eisenstein newforms. In particular, the full period
polynomial of
$E_6(163\tau;\chi_{163},\chi_{163})$ has four nonzero zeros off
the circle $|X|=1/163$, although this Eisenstein series
is a Hecke eigenform, new of level $163^2$, with trivial
Nebentypus. Thus the Hecke eigenform and newform conditions
alone do not imply the zero-location theorem known for
cuspidal newforms.\\

Theorem~\ref{eventualzeros} describes the eventual behavior
for each fixed character in its hypotheses. If $I_\chi>0$,
all nonzero zeros lie on $\Sone$ in every sufficiently
large degree. If $I_\chi<0$, exactly four nonzero zeros lie off
$\Sone$, and these converge to four distinct points off
the unit circle. A natural question is therefore
to understand how the sign of $I_\chi$ varies with $\chi$.\\

Computations suggest that $I_\chi<0$ occurs with
a positive limiting proportion both among odd primitive
real characters, ordered by conductor, and among even
imprimitive real nonprincipal characters, ordered by
modulus. In the latter family, each induced character
is counted separately at each modulus. Up to $10^6$,
the respective proportions are
\[
\frac{7025}{303968}\approx2.3111\%,
\qquad
\frac{1995}{3422260}\approx0.0583\%.
\]
We expect both limiting proportions to lie strictly
between $0$ and $1$. Proving their existence and
positivity, and determining their values, remain open.
Such a result would show that a positive proportion
of characters in each family give four non-unimodular
nonzero zeros in every sufficiently large degree.\\

A further problem is to obtain useful upper bounds for the threshold $n_0(\chi)$
in Theorem~\ref{eventualzeros} explicit in terms of
arithmetic data. The proof gives
\[
H_d(-1)=2I_\chi+O_\chi\left( \frac{d}{4^{d}} \right)
\qquad(d\ \text{even}),
\]
and
\[
\frac{H_d'(-1)}{d}
=I_\chi+O_\chi\left( \frac{1}{d}\right)
\qquad(d\ \text{odd}).
\]
Tracking the constants yields an effective threshold
once a positive lower bound for $|I_\chi|$ is available. Such bounds require quantitative control of $|I_\chi|$ in terms of the conductor and the primes dividing the modulus but not the conductor. Understanding the sign and size of $I_\chi$ therefore remains a natural direction for further investigation.

\section*{AI disclosure}
ChatGPT (non-pro versions; models GPT-5.6 Sol and GPT-6 Astra) was used to assist with exploratory investigations and the computation and verification of examples. All numerical bounds used in the proofs are established by explicit theoretical arguments. The author takes responsibility for the mathematical results and arguments presented in this paper.

\section*{Acknowledgments}
The author would like to thank Mrityunjoy Charan and Asif Zaman for helpful discussions.

\end{document}